\documentclass[12pt]{amsart}
\usepackage[dvips,centering,includehead,width=15.1cm,height=22.7cm]{geometry}

\usepackage{amsfonts, amsmath, amssymb, amsthm, pst-node, epsfig, enumerate}
\usepackage{pstricks}

\theoremstyle{plain}
\newtheorem{thm}{Theorem}[section]

\newtheorem{cor}[thm]{Corollary}

\newtheorem{lemma}[thm]{Lemma}

\theoremstyle{definition}
\newtheorem{defn}[thm]{Definition}
\newtheorem{example}[thm]{Example}

\numberwithin{equation}{section}

\def\wt{\operatorname{wt}}

\def\Meas{\operatorname{Meas}}

\def\Gr{\operatorname{Gr}}
\def\Nest{\operatorname{Nest}}
\def\lexmax{\operatorname{lexmax}}

\def\Le{\hbox{\rotatedown{$\Gamma$}}}
\def\P{\mathcal{P}}
\def\T{\mathbf{T}}
\def\M{\mathcal{M}}

\def\posit{(S_\M)_{\geq0}}
\def\positL{(S_{\M_L})_{\geq0}}

\begin{document}

\title[$\Le$-coordinates in a totally nonnegative Grassmannian]{Combinatorial formulas for $\Le$-coordinates\\ in a totally nonnegative Grassmannian}
\author{Kelli Talaska}
\address{Department of Mathematics, University of Michigan,
Ann Arbor, MI 48109, USA} \email{kellicar@umich.edu}
\thanks{The author was supported by NSF Grants DMS-0502170 and DMS-0555880.}

\date{\today}



\begin{abstract}
Postnikov constructed a decomposition of a totally nonnegative Grassmannian $(\Gr_{kn})_{\geq0}$ into positroid cells.
We provide combinatorial formulas that allow one to decide which cell a given point in $(\Gr_{kn})_{\geq0}$ belongs to and to determine affine coordinates of the point within this cell. This simplifies Postnikov's description of the inverse boundary measurement map and generalizes formulas for the top cell given by Speyer and Williams. In addition, we identify a particular subset of Pl\"ucker coordinates as a totally positive base for the set of non-vanishing Pl\"ucker coordinates for a given positroid cell.
\end{abstract}

\maketitle

Postnikov \cite{Postnikov2007} has described a cell decomposition of a totally nonnegative Grassmannian into positroid cells, which are indexed by $\Le$-diagrams; this decomposition is analogous to the matroid stratification of a real Grassmannian given by Gel'fand, Goresky, MacPherson, and Serganova \cite{GGMS1987}.  Postnikov also introduced a parametrization of each positroid cell using a collection of parameters which we call $\Le$-coordinates.

In this paper, we give an explicit criterion for determining which positroid cell contains a given point in a totally nonnegative Grassmannian and explicit combinatorial formulas for the $\Le$-coordinates of a point.  This generalizes the formulas of Speyer and Williams given for the top dimensional positroid cell \cite{SW2005}, and provides a simpler description of Postnikov's inverse boundary measurement map, which was given recursively in \cite{Postnikov2007}.  For a fixed positroid cell, our formulas are written in terms of a minimal set of Pl\"ucker coordinates, and this minimal set forms a totally positive base (in the sense of Fomin and Zelevinsky \cite{FZ1999}) for the set of Pl\"ucker coordinates which do not vanish on the specified cell.

\section{Positroid stratification and the boundary measurement map}\label{sec:defs}

In this section, we review Postnikov's positroid stratification of a totally nonnegative Grassmannian and boundary measurement map.

Let $\Gr_{kn}$ denote the Grassmannian of $k$-dimensional subspaces of $\mathbb{R}^n$.  A point $x\in \Gr_{kn}$ can be described by a collection of (projective) Pl\"ucker coordinates $(P_J(x))$, indexed by $k$-element subsets $J\subset[n]$.  The \emph{totally nonnegative Grassmannian $(\Gr_{kn})_{\geq0}$} is the subset of points  $x\in\Gr_{kn}$ such that all Pl\"ucker coordinates $P_J(x)$ can be chosen to be simultaneously nonnegative.

In \cite{GGMS1987}, the authors gave a decomposition of the Grassmannian $\Gr_{kn}$ into \emph{matroid strata}.  More precisely, for a matroid $\M\subseteq {[n]\choose k}$, let $S_\M$ denote the subset of points $x\in \Gr_{kn}$ such that $P_J(x)\neq0$ if and only if $J\in \M$.  In particular, each possible vanishing pattern of Pl\"ucker coordinates is given by a unique (realizable) matroid $\M$.  In \cite{Postnikov2007}, Postnikov studies a natural analogue of the matroid stratification for the totally nonnegative Grassmannian, a decomposition into disjoint \emph{positroid cells} of the form $\posit=S_\M\cap (\Gr_{kn})_{\geq0}$.

\begin{defn}
A \emph{$\Le$-diagram} is a partition $\lambda$ together with a filling of the boxes of the Young diagram of $\lambda$ with entries $0$ and $+$ satisfying the $\Le$-property: there is no $0$ which has a $+$ above it (in the same column) and a $+$ to its left (in the same row).

Replacing the boxes labeled $+$ in a $\Le$-diagram with positive real numbers, called \emph{$\Le$-coordinates}, we obtain a \emph{$\Le$-tableau}.  Let $\T_{L}$ denote the set of $\Le$-tableaux whose vanishing pattern is given by the $\Le$-diagram $L$.  Note that $\T_L$ is an affine space whose dimension is equal to the number of ``+'' entries in $L$, which we denote by $|L|$.

For a box $B$ in $\lambda$, we let $L_{B}$ and $T_{B}$ denote the labels of the box $B$ in the $\Le$-diagram $L$ and the $\Le$-tableau $T$, respectively.
\end{defn}

In the positroid cell decomposition of $(\Gr_{kn})_{\geq0}$ given in \cite{Postnikov2007}, the positroid cells are indexed by $\Le$-diagrams $L$ which fit inside a $k\times (n-k)$ rectangle.  Further, the positroid corresponding to a fixed $\Le$-diagram $L$ is parametrized by the $\Le$-tableaux $T\in\T_L$, i.e., those whose vanishing pattern is given by $L$.

The parametrization described below is a special case of Postnikov's boundary measurement map.  To give a formula for this parametrization, we need to introduce certain planar networks, called $\Gamma$-networks, which are in bijection with $\Le$-tableaux.

\begin{figure}[ht]
\psset{unit=0.9cm}
\begin{pspicture*}(-12.25,-1)(5,6)
\psset{dotstyle=*,dotsize=5pt 0,linewidth=0.8pt,arrowsize=3pt 2,arrowinset=0.25}
\psline[linewidth=1.6pt,arrows=*-*](-3,5)(4,5)
\psline[linewidth=1.6pt,linestyle=dashed,dash=5pt 5pt](0.85,0)(-3,0)
\psline[linewidth=1.6pt,arrows=*-*](-3,0)(-3,5)
\psline[linewidth=1.6pt,arrows=*-*](4,4)(0,4)
\psline[linewidth=1.6pt,arrows=*-*](0,4)(0,0)
\psline[linewidth=1.6pt,arrows=*-*](-2,0)(-2,3)
\psline[linewidth=1.6pt,arrows=*-*](-2,3)(4,3)
\psline[linewidth=1.6pt,arrows=*-*](-1,2)(-1,0)
\psline[linewidth=1.6pt,linestyle=dashed,dash=5pt 5pt](-3,0)(-3.5,0)
\psline[linewidth=1.6pt,linestyle=dashed,dash=5pt 5pt](-3.5,0)(-3.5,5.5)
\psline[linewidth=1.6pt,linestyle=dashed,dash=5pt 5pt](-3.5,5.5)(4,5.5)
\psline[linewidth=1.6pt,linestyle=dashed,dash=5pt 5pt](4,5.5)(4,5)
\psline[linewidth=1.6pt,linestyle=dashed,dash=5pt 5pt](0.85,1.15)(0.85,0)
\psline[linewidth=1.6pt,linestyle=dashed,dash=5pt 5pt](0.85,1.15)(2.85,1.15)
\psline[linewidth=1.6pt,linestyle=dashed,dash=5pt 5pt](2.85,1.15)(2.85,2.15)
\psline[linewidth=1.6pt,linestyle=dashed,dash=5pt 5pt](2.85,2.15)(4,2.15)
\psline[linewidth=1.6pt,arrows=*-*](-3,1)(0.85,1)
\psline[linewidth=1.6pt,arrows=*-*](1,3)(1,1.15)
\psline[linewidth=1.6pt,arrows=*-*](-1,2)(2.85,2)
\psline[linewidth=1.6pt,arrows=*-*](3,3)(3,2.15)
\psline[linewidth=1.6pt,linestyle=dashed,dash=5pt 5pt](4,5)(4,2.15)
\psline[linewidth=1.6pt,arrows=*-*](0,3)(0,3)\psline[linewidth=1.6pt,arrows=*-*](0,2)(0,2)
\psline[linewidth=1.6pt,arrows=*-*](0,1)(0,1)\psline[linewidth=1.6pt,arrows=*-*](-1,1)(-1,1)
\psline[linewidth=1.6pt,arrows=*-*](-2,1)(-2,1)\psline[linewidth=1.6pt,arrows=*-*](1,2)(1,2)

\psline[linewidth=1.2pt](0.5,5)(0.65,4.85)
\psline[linewidth=1.2pt](0.5,5)(0.65,5.15)
\psline[linewidth=1.2pt](2,4)(2.15,3.85)
\psline[linewidth=1.2pt](2,4)(2.15,4.15)
\psline[linewidth=1.2pt](-1,3)(-0.85,3.15)
\psline[linewidth=1.2pt](-1,3)(-0.85,2.85)
\psline[linewidth=1.2pt](0.5,3)(0.65,3.15)
\psline[linewidth=1.2pt](0.5,3)(0.65,2.85)
\psline[linewidth=1.2pt](2,3)(2.15,3.15)
\psline[linewidth=1.2pt](2,3)(2.15,2.85)
\psline[linewidth=1.2pt](3.5,3)(3.65,3.15)
\psline[linewidth=1.2pt](3.5,3)(3.65,2.85)
\psline[linewidth=1.2pt](-0.5,2)(-0.35,2.15)
\psline[linewidth=1.2pt](-0.5,2)(-0.35,1.85)
\psline[linewidth=1.2pt](0.5,2)(0.65,2.15)
\psline[linewidth=1.2pt](0.5,2)(0.65,1.85)
\psline[linewidth=1.2pt](2,2)(2.15,2.15)
\psline[linewidth=1.2pt](2,2)(2.15,1.85)
\psline[linewidth=1.2pt](-2.5,1)(-2.35,1.15)
\psline[linewidth=1.2pt](-2.5,1)(-2.35,0.85)
\psline[linewidth=1.2pt](-1.5,1)(-1.35,1.15)
\psline[linewidth=1.2pt](-1.5,1)(-1.35,0.85)
\psline[linewidth=1.2pt](-0.5,1)(-0.35,1.15)
\psline[linewidth=1.2pt](-0.5,1)(-0.35,0.85)
\psline[linewidth=1.2pt](0.5,1)(0.65,1.15)
\psline[linewidth=1.2pt](0.5,1)(0.65,0.85)

\psline[linewidth=1.4pt](-3,3)(-2.85,3.15)
\psline[linewidth=1.4pt](-3,3)(-3.15,3.15)
\psline[linewidth=1.4pt](-3,0.5)(-2.85,0.65)
\psline[linewidth=1.4pt](-3,0.5)(-3.15,0.65)
\psline[linewidth=1.4pt](-2,2)(-1.85,2.15)
\psline[linewidth=1.4pt](-2,2)(-2.15,2.15)
\psline[linewidth=1.4pt](-2,0.5)(-1.85,0.65)
\psline[linewidth=1.4pt](-2,0.5)(-2.15,0.65)
\psline[linewidth=1.4pt](-1,1.5)(-0.85,1.65)
\psline[linewidth=1.4pt](-1,1.5)(-1.15,1.65)
\psline[linewidth=1.4pt](-1,0.5)(-0.85,0.65)
\psline[linewidth=1.4pt](-1,0.5)(-1.15,0.65)
\psline[linewidth=1.4pt](0,3.5)(-0.15,3.65)
\psline[linewidth=1.4pt](0,3.5)(0.15,3.65)
\psline[linewidth=1.4pt](0,2.5)(-0.15,2.65)
\psline[linewidth=1.4pt](0,2.5)(0.15,2.65)
\psline[linewidth=1.4pt](0,1.5)(-0.15,1.65)
\psline[linewidth=1.4pt](0,1.5)(0.15,1.65)
\psline[linewidth=1.4pt](0,0.5)(-0.15,0.65)
\psline[linewidth=1.4pt](0,0.5)(0.15,0.65)
\psline[linewidth=1.4pt](1,2.5)(0.85,2.65)
\psline[linewidth=1.4pt](1,2.5)(1.15,2.65)
\psline[linewidth=1.4pt](1,1.5)(0.85,1.65)
\psline[linewidth=1.4pt](1,1.5)(1.15,1.65)
\psline[linewidth=1.4pt](3,2.5)(2.85,2.65)
\psline[linewidth=1.4pt](3,2.5)(3.15,2.65)

\rput[tl](4.2,5.1){$1$}
\rput[tl](4.2,4.1){$2$}
\rput[tl](4.2,3.1){$3$}
\rput[tl](0.95,0.8){$8$}
\rput[tl](-0.1,-0.1){$9$}
\rput[tl](-2.2,-0.1){$11$}
\rput[tl](-3.22,-0.1){$12$}
\rput[tl](-1.16,-0.1){$10$}
\rput[tl](3.2,2.05){$4$}
\rput[tl](2.95,1.8){$5$}
\rput[tl](1.92,1){$6$}
\rput[tl](1.2,1){$7$}
\psline[linewidth=1.6pt,linestyle=dashed,dash=3pt 3pt](-3,0)(-3.5,0)
\psline[linewidth=1.6pt,linestyle=dashed,dash=3pt 3pt](-3.5,0)(-3.5,5.5)
\psline[linewidth=1.6pt,linestyle=dashed,dash=3pt 3pt](-3.5,5.5)(4,5.5)
\psline[linewidth=1.6pt,linestyle=dashed,dash=3pt 3pt](4,5.5)(4,5)

\rput[tl](-2.75,4.65){$T_{17}$}
\rput[tl](0.25,3.65){$T_{24}$}
\rput[tl](-1.75,2.65){$T_{36}$}
\rput[tl](0.25,2.65){$T_{34}$}
\rput[tl](1.25,2.65){$T_{33}$}
\rput[tl](3.25,2.65){$T_{31}$}
\rput[tl](-0.75,1.65){$T_{45}$}
\rput[tl](0.25,1.65){$T_{44}$}
\rput[tl](1.25,1.65){$T_{43}$}
\rput[tl](-2.75,0.65){$T_{57}$}
\rput[tl](-1.75,0.65){$T_{56}$}
\rput[tl](-0.75,0.65){$T_{55}$}
\rput[tl](0.25,0.65){$T_{54}$}

\psline[linewidth=1.2pt](-12,5)(-5,5)
\psline[linewidth=1.2pt](-5,5)(-5,2)
\psline[linewidth=1.2pt](-5,2)(-6,2)
\psline[linewidth=1.2pt](-6,2)(-6,1)
\psline[linewidth=1.2pt](-6,1)(-8,1)
\psline[linewidth=1.2pt](-8,1)(-8,0)
\psline[linewidth=1.2pt](-8,0)(-12,0)
\psline[linewidth=1.2pt](-12,0)(-12,1)
\psline[linewidth=1.2pt](-12,1)(-12,2)
\psline[linewidth=1.2pt](-12,2)(-12,3)
\psline[linewidth=1.2pt](-12,3)(-12,4)
\psline[linewidth=1.2pt](-12,4)(-12,5)
\psline[linewidth=1.2pt](-11,5)(-11,0)
\psline[linewidth=1.2pt](-10,0)(-10,5)
\psline[linewidth=1.2pt](-9,5)(-9,0)
\psline[linewidth=1.2pt](-8,5)(-8,1)
\psline[linewidth=1.2pt](-7,5)(-7,1)
\psline[linewidth=1.2pt](-6,5)(-6,2)
\psline[linewidth=1.2pt](-5,4)(-12,4)
\psline[linewidth=1.2pt](-12,3)(-5,3)
\psline[linewidth=1.2pt](-6,2)(-12,2)
\psline[linewidth=1.2pt](-12,1)(-8,1)
\rput[tl](-11.75,4.6){$T_{17}$}
\rput[tl](-8.75,3.6){$T_{24}$}
\rput[tl](-10.75,2.6){$T_{36}$}
\rput[tl](-5.75,2.6){$T_{31}$}
\rput[tl](-8.75,2.6){$T_{34}$}
\rput[tl](-9.75,1.6){$T_{45}$}
\rput[tl](-8.75,1.6){$T_{44}$}
\rput[tl](-10.75,0.6){$T_{56}$}
\rput[tl](-11.75,0.6){$T_{57}$}
\rput[tl](-8.75,0.6){$T_{54}$}
\rput[tl](-9.75,0.6){$T_{55}$}
\rput[tl](-7.75,1.6){$T_{43}$}
\rput[tl](-7.75,2.6){$T_{33}$}
\rput[tl](-10.65,4.6){$0$}
\rput[tl](-9.65,4.6){$0$}
\rput[tl](-8.65,4.6){$0$}
\rput[tl](-7.65,4.6){$0$}
\rput[tl](-6.65,4.6){$0$}
\rput[tl](-5.65,4.6){$0$}
\rput[tl](-11.65,3.6){$0$}
\rput[tl](-10.65,3.6){$0$}
\rput[tl](-9.65,3.6){$0$}
\rput[tl](-7.65,3.6){$0$}
\rput[tl](-6.65,3.6){$0$}
\rput[tl](-5.65,3.6){$0$}
\rput[tl](-11.65,2.6){$0$}
\rput[tl](-9.65,2.6){$0$}
\rput[tl](-6.65,2.6){$0$}
\rput[tl](-11.65,1.6){$0$}
\rput[tl](-10.65,1.6){$0$}
\rput[tl](-6.65,1.6){$0$}
\end{pspicture*}
\caption{The $\Le$-tableau $T$ and $\Gamma$-network $N_T$ for a point in $(Gr_{5,12})_{\geq0}$.  We have $\lambda=(7,7,7,6,4)$ and $I=\{1,2,3,5,8\}$.}\label{fig:main}
\end{figure}
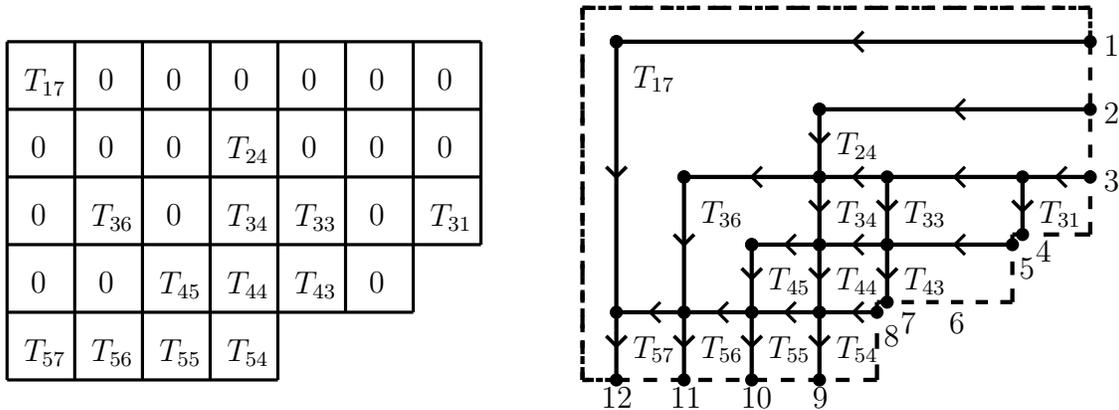

For each $\Le$-diagram $L$ of shape $\lambda$ which fits inside a $k\times(n-k)$ rectangle, we will construct a $\Gamma$-graph $G_L$ corresponding to $L$.  For each $\Le$-tableau $T\in\T_L$, we will then assign weights to the faces of $G_L$ to obtain a $\Gamma$-network $N_T$.

We begin by establishing the boundary of the planar network.  First, we draw a disk whose boundary consists of the north and west edges of the $k\times (n-k)$ box and the path determining the southeast boundary of $\lambda$, all shifted slightly northwest.  Place a vertex, called a \emph{boundary source}, at the end of each row (including empty rows) of $\lambda$, and a vertex, called a \emph{boundary sink}, at the end of each column of $\lambda$ (including empty columns).  Label these in sequence with the integers $\{1,2,\ldots, n\}$, following the path from the northeast corner to the southwest corner which determines $\lambda$.  Let $I=\{i_1< i_2< \cdots< i_k\}\subset [n]$ be the set of boundary sources, so that the complement of $I$, $[n]\setminus I=\{j_1<j_2<\cdots<j_{n-k}\}$, is the set of boundary sinks.

Whenever $L_{B}=+$, we draw the \emph{$B$ hook}, i.e., the hook whose corner is the northwest corner of the box $B=(r,c)$ (in the $r^{\rm{th}}$ row from the top and the $c^{\rm{th}}$ column from the right) and which has a horizontal path directed from the boundary source $i_r$ to the corner and a vertical path directed from the corner to the boundary sink $j_c$.  This process yields a $\Gamma$-graph $G_L$.

To obtain the $\Gamma$-network $N_T$ from $G_L$, we must assign weights to each of the faces.  Note that there is exactly one face for each box $B$ in $\lambda$ satisfying $L_{B}=+$ (and this face has a portion of the $B$ hook as its northwest boundary), and in addition, there is one face whose northwest boundary is the boundary of the disk.  For each box $B$ with $L_{B}=+$, we assign to the corresponding face the positive real weight $T_{B}$.  To the face whose northwest boundary is the boundary of the disk, we assign the weight $\prod\frac{1}{T_{B}}$, taking the product over boxes satisfying $L_{B}=+$, so that the product of all face weights in $N$ is exactly $1$.\\

In the special case of $\Gamma$-networks, the definition of Postnikov's map given in \cite{Postnikov2007} can be viewed as an instance of the classical formula of Lindstr\"om \cite{Lindstrom1973}.  This formula is usually given in terms of weights of edges; we apply Postnikov's transformation from edge weights to face weights \cite{Postnikov2007} to obtain the following restatement of his definition.

\begin{defn}\label{def:gamma-meas-map}
For each $\Le$-diagram $L$ which fits in a $k\times (n-k)$ rectangle, the \emph{boundary measurement map} $\Meas_L:\T_{L}\rightarrow (\Gr_{kn})_{\geq0}$ is defined by
\begin{equation*}
\label{eq:main-formula}
P_J(\Meas_L(T))=\sum_{A\in\mathcal{A}_J(N_T)} \wt(A),
\end{equation*}
where \begin{itemize}
\item $N_T$ is the $\Gamma$-network corresponding to the $\Le$-tableau $T$, and its boundary source set is labeled by $I$,
\item $\mathcal{A}_J(N_T)$ is the collection of non-intersecting path families $A=\{A_i\}_{i\in I}$ in $N_T$ from the boundary sources $I$ to the boundary destinations $J$,
\item $\wt(A)=\prod_{i\in I} \wt(A_i)$, and
\item the weight $\wt(A_i)$ of a path $A_i$ in the family $A$ is the product of the weights of the faces of $N_T$ which lie southeast of $A_i$.\end{itemize}
\end{defn}

For a $\Le$-diagram $L$, let $G_L$ be the corresponding $\Gamma$-graph.  Let us define the set $\M_L\subseteq{[n]\choose k}$ by the condition that $J\in \M_L$ if and only if there exists a non-intersecting path collection in $G_L$ with sources $I$ and destinations $J$.  It can be shown that $\M_L$ has the structure of a matroid, but this is not necessary for our purposes. Further, it is easily verified that for distinct $\Le$-diagrams $L$ and $L^*$, we have $\M_L\neq\M_{L^*}$.

\begin{thm} \cite{Postnikov2007} \label{thm:postnikov-bijection}
For each $\Le$-diagram $L$ which fits in a $k\times (n-k)$ rectangle, the map $\Meas_L:\T_{L}\rightarrow (\Gr_{kn})_{\geq0}$ is injective, and the image $\Meas_L(\T_L)$ is precisely the positroid cell $\positL$.

These positroid cells are pairwise disjoint, and the union $\bigcup_L \positL$, taken over all $\Le$-diagrams $L$ which fit inside the $k\times(n-k)$ rectangle, is the entire totally nonnegative Grassmannian $(\Gr_{kn})_{\geq0}$.  Each positroid cell $\positL$ is  a topological cell; that is, $\positL$ is isomorphic to $\mathbb{R}^{|L|}$, where $|L|$ is the number of ``+'' entries in $L$.  Thus, the positroid cells form a cell decomposition of $(\Gr_{kn})_{\geq0}$.
\end{thm}

In Postnikov's work \cite{Postnikov2007}, this result is proved by giving a recursive algorithm for finding the $\Le$-tableau $T$ corresponding to a given point in $(\Gr_{kn})_{\geq0}$.  In this paper, we obtain explicit combinatorial formulas solving the same problem.  This is done in two stages.  In Section~\ref{sec:le-diagram}, we give an explicit rule for determining which positroid cell contains a given point.  In Sections~\ref{sec:mobius}~and~\ref{sec:minimal}, we give two combinatorial formulas for the inverse of each particular map $\Meas_L$ (i.e., formulas for the corresponding $\Le$-coordinates) in terms of the relevant Pl\"ucker coordinates.

\section{Determining the positroid cell of a point in $(\Gr_{kn})_{\geq0}$}\label{sec:le-diagram}

In this section, we give an explicit formula for the $\Le$-tableau $L(x)$ that determines which positroid cell $\positL$ a given point $x\in(\Gr_{kn})_{\geq0}$ belongs to.

Let $x\in(\Gr_{kn})_{\geq0}$ be given by its Pl\"ucker coordinates $(P_J(x):J\in{{[n]}\choose k})$.  Order the $k$-subsets of $[n]$ lexicographically.  That is, a set $A=\{a_1<a_2<\cdots<a_k\}$ is less than or equal to a set $B=\{b_1<b_2<\cdots<b_k\}$ if at the smallest index $m$ for which $a_m\neq b_m$, we have $a_m<b_m$.

For $\M\subseteq {[n]\choose k}$, let $I=\{i_1<i_2<\cdots < i_k\}$ be the lexicographically minimum set in $\M$. Let $[n]\setminus I=\{j_1<j_2<\cdots<j_{n-k}\}$ be the complement of $I$.  Let $\lambda(\M)$ be the partition in the $k\times(n-k)$ rectangle whose southeastern border is given by the path from the northeast corner of the $k\times (n-k)$ rectangle to its southwest corner which has edges to the south in positions $I$ and edges to the west in positions $[n]\setminus I$.  More precisely, the length of the $t^{th}$ row of $\lambda$ is the number of elements of $[n]\setminus I$ which are greater than $i_t$, i.e., $\lambda_t=|j_s\in [n]\setminus I:j_s>i_t|$.  We let $A_{r,c}=[n]\setminus\{i_r+1, i_r+2, \ldots, j_c-2,j_c-1\}$.  Set $M(B,\M)=\left(M'(B,\M)\setminus\{i_r\}\right)\cup\{j_c\}$, where
\[
M'(B,\M)=\lexmax\left\{J\in \M(x): J\cap A_{r,c}=I\cap A_{r,c}\right\}.
\]

In plain language, this says that we are taking the maximum over sets $J$ which contain all of the sources outside the open interval from $i_r$ to $j_c$ and none of the sinks, i.e., those sets whose interesting behavior happens \emph{inside} the interval.

Recall that for a $\Le$-diagram $L$, we have $J\in \M_L$ if and only if there exists a non-intersecting path collection in $G_L$ with destination set $J$.

\begin{lemma}\label{lem:nw-path-coll}
Suppose that $B$ is a box in a $\Le$-diagram $L$ of shape $\lambda(L)$.  Then
\begin{enumerate}
\item $M'(B,\M_L)$ is the destination set of a unique non-intersecting path collection in the $\Gamma$-graph $G_L$, namely the northwest-most path collection lying strictly southeast of the $B$ hook,
\item $M(B,\M_L)\in \M_L$ if and only if $L_B=+$, and
\item the vanishing pattern for the Pl\"ucker coordinates of $\positL$ is uniquely determined by the vanishing pattern of the subset $\{P_{M(B,\M_L)}\}$, ranging over all boxes $B$ in $\lambda(L)$.
\end{enumerate}
\end{lemma}

\begin{proof}
The proof of the first claim is left as a straightforward exercise for the reader; the second and third then follow immediately from the definitions.
\end{proof}

\begin{example}
 On the left in Figure~\ref{fig:addahook}, we have the $\Gamma$-graph of the example in Figure~\ref{fig:main}. We see that $M'((2,6),\M_L)=\{1,2,7,9,10\}$, corresponding to the solid path collection on the right in Figure~\ref{fig:addahook}. Adding in the potential (dotted) hook from $i_2=2$ to $j_6=11$, we have $M((2,6),\M_L)=\{1,7,9,10,11\}$.  Since this hook does not occur in the $\Gamma$-graph, we must have $P_{M((2,6),\M_L)}(x)=0$ for this point.

\begin{figure}[ht]
\psset{unit=0.6cm}
\begin{pspicture*}(-4,-1)(5,6)
\psset{dotstyle=*,dotsize=3pt 0,linewidth=0.8pt,arrowsize=3pt 2,arrowinset=0.25}
\psline[linewidth=1.4pt,arrows=*-*](-3,5)(4,5)
\psline[linewidth=1.4pt,linestyle=dashed,dash=5pt 5pt](0.85,0)(-3,0)
\psline[linewidth=1.4pt,arrows=*-*](-3,0)(-3,5)
\psline[linewidth=1.4pt,arrows=*-*](4,4)(0,4)
\psline[linewidth=1.4pt,arrows=*-*](0,4)(0,0)
\psline[linewidth=1.4pt,arrows=*-*](-2,0)(-2,3)
\psline[linewidth=1.4pt,arrows=*-*](-2,3)(4,3)
\psline[linewidth=1.4pt,arrows=*-*](-1,2)(-1,0)
\psline[linewidth=1.4pt,linestyle=dashed,dash=5pt 5pt](-3,0)(-3.5,0)
\psline[linewidth=1.4pt,linestyle=dashed,dash=5pt 5pt](-3.5,0)(-3.5,5.5)
\psline[linewidth=1.4pt,linestyle=dashed,dash=5pt 5pt](-3.5,5.5)(4,5.5)
\psline[linewidth=1.4pt,linestyle=dashed,dash=5pt 5pt](4,5.5)(4,5)
\psline[linewidth=1.4pt,linestyle=dashed,dash=5pt 5pt](0.85,1.15)(0.85,0)
\psline[linewidth=1.4pt,linestyle=dashed,dash=5pt 5pt](0.85,1.15)(2.85,1.15)
\psline[linewidth=1.4pt,linestyle=dashed,dash=5pt 5pt](2.85,1.15)(2.85,2.15)
\psline[linewidth=1.4pt,linestyle=dashed,dash=5pt 5pt](2.85,2.15)(4,2.15)
\psline[linewidth=1.4pt,arrows=*-*](-3,1)(0.85,1)
\psline[linewidth=1.4pt,arrows=*-*](1,3)(1,1.15)
\psline[linewidth=1.4pt,arrows=*-*](-1,2)(2.85,2)
\psline[linewidth=1.4pt,arrows=*-*](3,3)(3,2.15)
\psline[linewidth=1.4pt,linestyle=dashed,dash=5pt 5pt](4,5)(4,2.15)
\psline[linewidth=1.6pt,arrows=*-*](0,3)(0,3)\psline[linewidth=1.6pt,arrows=*-*](0,2)(0,2)
\psline[linewidth=1.6pt,arrows=*-*](0,1)(0,1)\psline[linewidth=1.6pt,arrows=*-*](-1,1)(-1,1)
\psline[linewidth=1.6pt,arrows=*-*](-2,1)(-2,1)\psline[linewidth=1.6pt,arrows=*-*](1,2)(1,2)

\psline[linewidth=1.2pt](0.5,5)(0.65,4.85)
\psline[linewidth=1.2pt](0.5,5)(0.65,5.15)
\psline[linewidth=1.2pt](2,4)(2.15,3.85)
\psline[linewidth=1.2pt](2,4)(2.15,4.15)
\psline[linewidth=1.2pt](-1,3)(-0.85,3.15)
\psline[linewidth=1.2pt](-1,3)(-0.85,2.85)
\psline[linewidth=1.2pt](0.5,3)(0.65,3.15)
\psline[linewidth=1.2pt](0.5,3)(0.65,2.85)
\psline[linewidth=1.2pt](2,3)(2.15,3.15)
\psline[linewidth=1.2pt](2,3)(2.15,2.85)
\psline[linewidth=1.2pt](3.5,3)(3.65,3.15)
\psline[linewidth=1.2pt](3.5,3)(3.65,2.85)
\psline[linewidth=1.2pt](-0.5,2)(-0.35,2.15)
\psline[linewidth=1.2pt](-0.5,2)(-0.35,1.85)
\psline[linewidth=1.2pt](0.5,2)(0.65,2.15)
\psline[linewidth=1.2pt](0.5,2)(0.65,1.85)
\psline[linewidth=1.2pt](2,2)(2.15,2.15)
\psline[linewidth=1.2pt](2,2)(2.15,1.85)
\psline[linewidth=1.2pt](-2.5,1)(-2.35,1.15)
\psline[linewidth=1.2pt](-2.5,1)(-2.35,0.85)
\psline[linewidth=1.2pt](-1.5,1)(-1.35,1.15)
\psline[linewidth=1.2pt](-1.5,1)(-1.35,0.85)
\psline[linewidth=1.2pt](-0.5,1)(-0.35,1.15)
\psline[linewidth=1.2pt](-0.5,1)(-0.35,0.85)
\psline[linewidth=1.2pt](0.5,1)(0.65,1.15)
\psline[linewidth=1.2pt](0.5,1)(0.65,0.85)

\psline[linewidth=1.4pt](-3,3)(-2.85,3.15)
\psline[linewidth=1.4pt](-3,3)(-3.15,3.15)
\psline[linewidth=1.4pt](-3,0.5)(-2.85,0.65)
\psline[linewidth=1.4pt](-3,0.5)(-3.15,0.65)
\psline[linewidth=1.4pt](-2,2)(-1.85,2.15)
\psline[linewidth=1.4pt](-2,2)(-2.15,2.15)
\psline[linewidth=1.4pt](-2,0.5)(-1.85,0.65)
\psline[linewidth=1.4pt](-2,0.5)(-2.15,0.65)
\psline[linewidth=1.4pt](-1,1.5)(-0.85,1.65)
\psline[linewidth=1.4pt](-1,1.5)(-1.15,1.65)
\psline[linewidth=1.4pt](-1,0.5)(-0.85,0.65)
\psline[linewidth=1.4pt](-1,0.5)(-1.15,0.65)
\psline[linewidth=1.4pt](0,3.5)(-0.15,3.65)
\psline[linewidth=1.4pt](0,3.5)(0.15,3.65)
\psline[linewidth=1.4pt](0,2.5)(-0.15,2.65)
\psline[linewidth=1.4pt](0,2.5)(0.15,2.65)
\psline[linewidth=1.4pt](0,1.5)(-0.15,1.65)
\psline[linewidth=1.4pt](0,1.5)(0.15,1.65)
\psline[linewidth=1.4pt](0,0.5)(-0.15,0.65)
\psline[linewidth=1.4pt](0,0.5)(0.15,0.65)
\psline[linewidth=1.4pt](1,2.5)(0.85,2.65)
\psline[linewidth=1.4pt](1,2.5)(1.15,2.65)
\psline[linewidth=1.4pt](1,1.5)(0.85,1.65)
\psline[linewidth=1.4pt](1,1.5)(1.15,1.65)
\psline[linewidth=1.4pt](3,2.5)(2.85,2.65)
\psline[linewidth=1.4pt](3,2.5)(3.15,2.65)

\rput[tl](4.2,5.1){$1$}
\rput[tl](4.2,4.1){$2$}
\rput[tl](4.2,3.1){$3$}
\rput[tl](0.95,0.8){$8$}
\rput[tl](-0.1,-0.1){$9$}
\rput[tl](-2.2,-0.1){$11$}
\rput[tl](-3.22,-0.1){$12$}
\rput[tl](-1.16,-0.1){$10$}
\rput[tl](3.2,2.05){$4$}
\rput[tl](2.95,1.8){$5$}
\rput[tl](1.92,1){$6$}
\rput[tl](1.2,1){$7$}
\psline[linewidth=1.6pt,linestyle=dashed,dash=3pt 3pt](-3,0)(-3.5,0)
\psline[linewidth=1.6pt,linestyle=dashed,dash=3pt 3pt](-3.5,0)(-3.5,5.5)
\psline[linewidth=1.6pt,linestyle=dashed,dash=3pt 3pt](-3.5,5.5)(4,5.5)
\psline[linewidth=1.6pt,linestyle=dashed,dash=3pt 3pt](4,5.5)(4,5)
\end{pspicture*}
\begin{pspicture*}(-4,-1)(5,6)
\psset{dotstyle=*,dotsize=3pt 0,linewidth=0.8pt,arrowsize=3pt 2,arrowinset=0.25}
\psline[linewidth=1.6pt,linestyle=dashed,dash=5pt 5pt](0.85,0)(-3,0)
\psline[linewidth=1.6pt](0,3)(0,2)
\psline[linewidth=1.6pt](0,1)(0,0)
\psline[linewidth=1.6pt](0,3)(4,3)
\psline[linewidth=1.6pt](-1,2)(-1,0)
\psline[linewidth=1.6pt,linestyle=dashed,dash=5pt 5pt](-3,0)(-3.5,0)
\psline[linewidth=1.6pt,linestyle=dashed,dash=5pt 5pt](-3.5,0)(-3.5,5.5)
\psline[linewidth=1.6pt,linestyle=dashed,dash=5pt 5pt](-3.5,5.5)(4,5.5)
\psline[linewidth=1.6pt,linestyle=dashed,dash=5pt 5pt](4,5.5)(4,5)
\psline[linewidth=1.6pt,linestyle=dashed,dash=5pt 5pt](0.85,1.15)(0.85,0)
\psline[linewidth=1.6pt,linestyle=dashed,dash=5pt 5pt](0.85,1.15)(2.85,1.15)
\psline[linewidth=1.6pt,linestyle=dashed,dash=5pt 5pt](2.85,1.15)(2.85,2.15)
\psline[linewidth=1.6pt,linestyle=dashed,dash=5pt 5pt](2.85,2.15)(4,2.15)
\psline[linewidth=1.6pt](0,1)(0.85,1)
\psline[linewidth=1.6pt](-1,2)(0,2)
\psline[linewidth=1.6pt](1,2)(2.85,2)
\psline[linewidth=1.6pt](1,2)(1,1.15)
\psline[linewidth=1.6pt,linestyle=dashed,dash=5pt 5pt](4,5)(4,2.15)
\rput[tl](4.2,5.1){$1$}
\rput[tl](4.2,4.1){$2$}
\rput[tl](4.2,3.1){$3$}
\rput[tl](0.95,0.8){$8$}
\rput[tl](-0.1,-0.1){$9$}
\rput[tl](-2.2,-0.1){$11$}
\rput[tl](-3.22,-0.1){$12$}
\rput[tl](-1.16,-0.1){$10$}
\rput[tl](3.2,2.05){$4$}
\rput[tl](2.95,1.8){$5$}
\rput[tl](1.92,1){$6$}
\rput[tl](1.2,1){$7$}
\psline[linewidth=1.6pt,linestyle=dashed,dash=3pt 3pt](-3,0)(-3.5,0)
\psline[linewidth=1.6pt,linestyle=dashed,dash=3pt 3pt](-3.5,0)(-3.5,5.5)
\psline[linewidth=1.6pt,linestyle=dashed,dash=3pt 3pt](-3.5,5.5)(4,5.5)
\psline[linewidth=1.6pt,linestyle=dashed,dash=3pt 3pt](4,5.5)(4,5)
\psline[linewidth=1.6pt,linestyle=dotted](-2,4)(-2,0)
\psline[linewidth=1.6pt,linestyle=dotted](-2,4)(4,4)

\end{pspicture*}
\caption{The $\Gamma$-graph of an example in $(\Gr_{5,12})_{\geq0}$ and the path families corresponding to $M'((2,6),\M_L)$ and $M((2,6),\M_L)$.}\label{fig:addahook}
\end{figure}
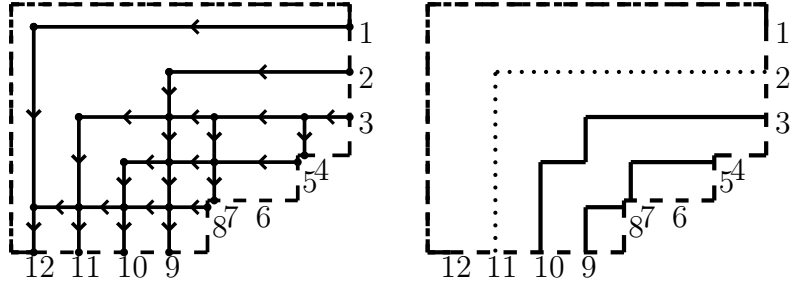
\end{example}

\begin{thm} \label{thm:le-diag}
For $x\in(\Gr_{kn})_{\geq0}$, set $\M(x)=\{J\in{[n]\choose k}: P_J(x)\neq0\}$.  Then the filling of $\lambda(\M(x))$ given by
\begin{equation}
L(x)_B=
\begin{cases}
0 & \text{if $P_{M(B,\M(x))}(x)=0$;}\\
+ & \text{if $P_{M(B,\M(x))}(x)\neq 0$.}
\end{cases}
\label{eq:le-diagram def}
\end{equation}
is a $\Le$-diagram, and $x$ lies in the positroid cell $\positL$.
\end{thm}

\begin{proof}
Combining Theorem~\ref{thm:postnikov-bijection} and Lemma~\ref{lem:nw-path-coll}, each point $x\in(\Gr_{kn})_{\geq0}$ lies in a unique positroid cell $\positL$ and therefore we must have a unique $\Le$-diagram $L$ such that $P_{M(B,\M(x))}=P_{M(B,\M_L)}$ for all boxes $B\in \lambda(L)=\lambda(\M(x))$.
\end{proof}

\section{The $\Le$-tableau associated with a point in $\positL$}\label{sec:mobius}

In Postnikov's original work, the map from $(\Gr_{kn})_{\geq0}$ to $\bigcup_L\T_{L}$ is given recursively.  In this section, we provide an explicit description of that map.  More precisely, given a point $x\in\positL$, we give combinatorial formulas for the entries of the parametrizing $\Le$-tableau, which we call \emph{$\Le$-coordinates} for $x$.

For each box $B$ in $\lambda$, let $H(B)$ denote the collection of boxes lying under the $B$ hook.  For each box $B$ with $L_B=+$, let $F(B)$ denote the face with northwest corner $B$, i.e., the collection of boxes which lie in the same face as $B$ in the $\Gamma$-graph $G$. We may simply write $F$ for $F(B)$ if there is no need to emphasize the northwest corner of $F$.  The $\Le$-property ensures that the northwest boundary of each face $F=F(B)$ is a portion of the $B$ hook; we may also call this the $F$ hook.

\begin{defn}
In a $\Gamma$-graph $G$, call a collection $W$ of paths a \emph{generalized path} if the paths in $W$ are pairwise disjoint, and no path of $W$ lies southeast of another path in $W$.

For a generalized path $W$ in a $\Gamma$-graph $G$, let $\mathcal{OC}(W)$ denote the set of \emph{outer corners} of $W$, that is, those boxes $B$ for which the northern and western boundaries of $B$ are both edges of $W$.  We order the outer corners from northeast to southwest.  Let $\mathcal{IC}(W)$ denote the \emph{inner corners} of $W$, that is, those boxes $B$ such the northwest boundary of $B$ is formed by portions of the hooks of two consecutive outer corners. Note that an inner corner need not be adjacent to the corresponding outer corners.  The $\Le$-property ensures that each outer or inner corner $B$ satisfies $L_{B}=+$.
\end{defn}

Let $D_F$ be the unique generalized path lying weakly under the $F$ hook which contains the entire southeast border of $F$. Informally, $\mathcal{OC}(D_F)$ indexes the hooks which determine the southeast boundary of the face $F$, and $\mathcal{IC}(D_F)$ indexes the hooks which are intersections of two hooks corresponding to adjacent outer corners.

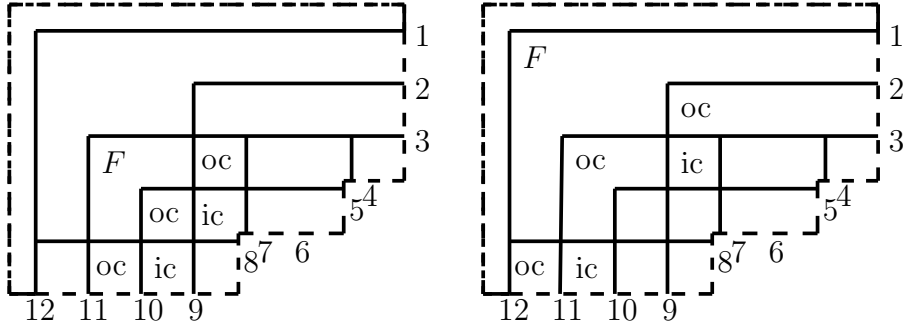
\begin{figure}[ht]
\psset{unit=0.7cm}
\begin{pspicture*}(-4,-1)(5,6)
\psset{dotstyle=*,dotsize=3pt 0,linewidth=0.8pt,arrowsize=3pt 2,arrowinset=0.25}
\psline[linewidth=1.4pt](-3,5)(4,5)
\psline[linewidth=1.4pt,linestyle=dashed,dash=5pt 5pt](0.85,0)(-3,0)
\psline[linewidth=1.4pt](-3,0)(-3,5)
\psline[linewidth=1.4pt](4,4)(0,4)
\psline[linewidth=1.4pt](0,4)(0,0)
\psline[linewidth=1.4pt](-2,0)(-2,3)
\psline[linewidth=1.4pt](-2,3)(4,3)
\psline[linewidth=1.4pt](-1,2)(-1,0)
\psline[linewidth=1.4pt,linestyle=dashed,dash=5pt 5pt](-3,0)(-3.5,0)
\psline[linewidth=1.4pt,linestyle=dashed,dash=5pt 5pt](-3.5,0)(-3.5,5.5)
\psline[linewidth=1.4pt,linestyle=dashed,dash=5pt 5pt](-3.5,5.5)(4,5.5)
\psline[linewidth=1.4pt,linestyle=dashed,dash=5pt 5pt](4,5.5)(4,5)
\psline[linewidth=1.4pt,linestyle=dashed,dash=5pt 5pt](0.85,1.15)(0.85,0)
\psline[linewidth=1.4pt,linestyle=dashed,dash=5pt 5pt](0.85,1.15)(2.85,1.15)
\psline[linewidth=1.4pt,linestyle=dashed,dash=5pt 5pt](2.85,1.15)(2.85,2.15)
\psline[linewidth=1.4pt,linestyle=dashed,dash=5pt 5pt](2.85,2.15)(4,2.15)
\psline[linewidth=1.4pt](-3,1)(0.85,1)
\psline[linewidth=1.4pt](1,3)(1,1.15)
\psline[linewidth=1.4pt](-1,2)(2.85,2)
\psline[linewidth=1.4pt](3,3)(3,2.15)
\psline[linewidth=1.4pt,linestyle=dashed,dash=5pt 5pt](4,5)(4,2.15)
\rput[tl](4.2,5.1){$1$}
\rput[tl](4.2,4.1){$2$}
\rput[tl](4.2,3.1){$3$}
\rput[tl](0.95,0.8){$8$}
\rput[tl](-0.1,-0.1){$9$}
\rput[tl](-2.2,-0.1){$11$}
\rput[tl](-3.22,-0.1){$12$}
\rput[tl](-1.16,-0.1){$10$}
\rput[tl](3.2,2.05){$4$}
\rput[tl](2.95,1.8){$5$}
\rput[tl](1.92,1){$6$}
\rput[tl](1.2,1){$7$}

\uput[r](-2,2.5){$F$}
\uput[r](-0.1,2.5){oc}
\uput[r](-1.1,1.5){oc}
\uput[r](-2.1,0.5){oc}
\uput[r](-0.1,1.5){ic}
\uput[r](-1,0.5){ic}
\psline[linewidth=1.4pt,linestyle=dashed,dash=3pt 3pt](-3,0)(-3.5,0)
\psline[linewidth=1.4pt,linestyle=dashed,dash=3pt 3pt](-3.5,0)(-3.5,5.5)
\psline[linewidth=1.4pt,linestyle=dashed,dash=3pt 3pt](-3.5,5.5)(4,5.5)
\psline[linewidth=1.4pt,linestyle=dashed,dash=3pt 3pt](4,5.5)(4,5)
\end{pspicture*}\begin{pspicture*}(-4,-1)(5,6)
\psset{dotstyle=*,dotsize=3pt 0,linewidth=0.8pt,arrowsize=3pt 2,arrowinset=0.25}
\psline[linewidth=1.4pt](-3,5)(4,5)
\psline[linewidth=1.4pt,linestyle=dashed,dash=5pt 5pt](0.85,0)(-3,0)
\psline[linewidth=1.4pt](-3,0)(-3,5)
\psline[linewidth=1.4pt](4,4)(0,4)
\psline[linewidth=1.4pt](0,4)(0,0)
\psline[linewidth=1.4pt](-2.04,0)(-2,3)
\psline[linewidth=1.4pt](-2,3)(4,3)
\psline[linewidth=1.4pt](-1,2)(-1,0)
\psline[linewidth=1.4pt,linestyle=dashed,dash=5pt 5pt](-3,0)(-3.5,0)
\psline[linewidth=1.4pt,linestyle=dashed,dash=5pt 5pt](-3.5,0)(-3.5,5.5)
\psline[linewidth=1.4pt,linestyle=dashed,dash=5pt 5pt](-3.5,5.5)(4,5.5)
\psline[linewidth=1.4pt,linestyle=dashed,dash=5pt 5pt](4,5.5)(4,5)
\psline[linewidth=1.4pt,linestyle=dashed,dash=5pt 5pt](0.85,1.15)(0.85,0)
\psline[linewidth=1.4pt,linestyle=dashed,dash=5pt 5pt](0.85,1.15)(2.85,1.15)
\psline[linewidth=1.4pt,linestyle=dashed,dash=5pt 5pt](2.85,1.15)(2.85,2.15)
\psline[linewidth=1.4pt,linestyle=dashed,dash=5pt 5pt](2.85,2.15)(4,2.15)
\psline[linewidth=1.4pt](-3,1)(0.85,1)
\psline[linewidth=1.4pt](1,3)(1,1.15)
\psline[linewidth=1.4pt](-1,2)(2.85,2)
\psline[linewidth=1.4pt](3,3)(3,2.15)
\psline[linewidth=1.4pt,linestyle=dashed,dash=5pt 5pt](4,5)(4,2.15)
\rput[tl](4.2,5.1){$1$}
\rput[tl](4.2,4.1){$2$}
\rput[tl](4.2,3.1){$3$}
\rput[tl](0.95,0.8){$8$}
\rput[tl](-0.1,-0.1){$9$}
\rput[tl](-2.2,-0.1){$11$}
\rput[tl](-3.22,-0.1){$12$}
\rput[tl](-1.16,-0.1){$10$}
\rput[tl](3.2,2.05){$4$}
\rput[tl](2.95,1.8){$5$}
\rput[tl](1.92,1){$6$}
\rput[tl](1.2,1){$7$}

\uput[r](-3,4.5){$F$}
\uput[r](0,3.5){oc}
\uput[r](-2,2.5){oc}
\uput[r](-3.155,0.5){oc}
\uput[r](0,2.5){ic}
\uput[r](-2,0.5){ic}
\psline[linewidth=1.4pt,linestyle=dashed,dash=3pt 3pt](-3,0)(-3.5,0)
\psline[linewidth=1.4pt,linestyle=dashed,dash=3pt 3pt](-3.5,0)(-3.5,5.5)
\psline[linewidth=1.4pt,linestyle=dashed,dash=3pt 3pt](-3.5,5.5)(4,5.5)
\psline[linewidth=1.4pt,linestyle=dashed,dash=3pt 3pt](4,5.5)(4,5)

\end{pspicture*}
\caption{Finding the corners of $D_F(3,6)$ and $D_F(1,7)$.}\label{fig:corners}
\end{figure}

\begin{example}
Consider the $\Gamma$-graph in Figure~\ref{fig:corners}.  We find the inner and outer corners of $D_{F(3,6)}$ and of $D_{F(1,7)}$.  In each graph, the relevant face is marked ``$F$'', outer corners are marked ``oc'', and inner corners are marked ``ic''.
\end{example}

We recall that the M\"obius function $\mu_S$ of a partially ordered set $S$ is given recursively by the rules
\begin{eqnarray*}
  \mu_S(x,x) &=&1, \text{ for all $x\in S$, and}\\
  \mu_S(x,y) &=&-\sum_{x\leq z<y}\mu_S(x,z), \text{ for all $x<y$ in $S$.}
\end{eqnarray*}

For a $\Le$-diagram $L$, let $\mathcal{F}_L$ denote the set of faces of the $\Gamma$-graph $G_L$ which are indexed by $+$ entries in $L$.  We partially order $\mathcal{F}_L$ by the condition that $F_1\leq_L F_2$ if the $F_1$ hook lies weakly northwest of the $F_2$ hook.

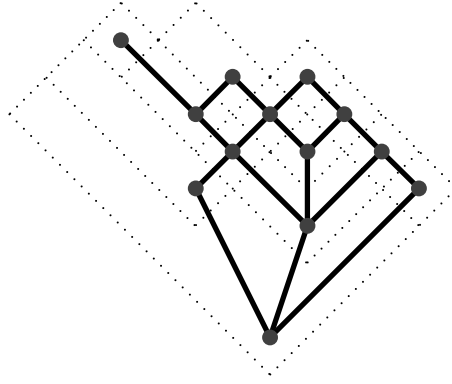
\begin{figure}[ht]
\newrgbcolor{xdxdff}{0.49 0.49 1}
\psset{xunit=0.7cm,yunit=0.7cm,dotstyle=*,dotsize=6pt 0,linewidth=0.8pt,arrowsize=3pt 2,arrowinset=0.25}
\begin{pspicture*}(-5.4,4.9)(3.8,12.3)
\psline[linestyle=dotted](0,5)(-4.95,9.95)
\psline[linestyle=dotted](-4.95,9.95)(-2.83,12.07)
\psline[linestyle=dotted](-2.83,12.07)(-2.12,11.36)
\psline[linestyle=dotted](-2.12,11.36)(-1.41,12.07)
\psline[linestyle=dotted](-1.41,12.07)(0,10.66)
\psline[linestyle=dotted](0,10.66)(0.71,11.36)
\psline[linestyle=dotted](0.71,11.36)(3.54,8.54)
\psline[linestyle=dotted](3.54,8.54)(0,5)
\psline[linewidth=2pt](-1.41,8.54)(0,5.71)
\psline[linewidth=2pt](0,5.71)(0.71,7.83)
\psline[linewidth=2pt](0,5.71)(2.83,8.54)
\psline[linewidth=2pt](2.83,8.54)(2.12,9.24)
\psline[linewidth=2pt](0.71,7.83)(2.12,9.24)
\psline[linewidth=2pt](0.71,7.83)(0.71,9.24)
\psline[linewidth=2pt](0.71,7.83)(-0.71,9.24)
\psline[linewidth=2pt](-0.71,9.24)(-1.41,8.54)
\psline[linewidth=2pt](0,9.95)(-0.71,9.24)
\psline[linewidth=2pt](0,9.95)(0.71,9.24)
\psline[linewidth=2pt](0.71,9.24)(1.41,9.95)
\psline[linewidth=2pt](1.41,9.95)(2.12,9.24)
\psline[linewidth=2pt](1.41,9.95)(0.71,10.66)
\psline[linewidth=2pt](0.71,10.66)(0,9.95)
\psline[linewidth=2pt](0,9.95)(-0.71,10.66)
\psline[linewidth=2pt](-0.71,9.24)(-1.41,9.95)
\psline[linewidth=2pt](-1.41,9.95)(-0.71,10.66)
\psline[linewidth=2pt](-1.41,9.95)(-2.83,11.36)
\psline(1,3)(1,0)
\psline(1,3)(7,3)
\psline(0,1)(4,1)
\psline(2,2)(2,0)
\psline(2,2)(6,2)
\psline(3,2)(3,0)
\psline(3,4)(3,2)
\psline(3,4)(7,4)
\psline(4,3)(4,1)
\psline(6,3)(6,2)
\psline[linestyle=dotted](0.7,7.12)(2.83,9.24)
\psline[linestyle=dotted](0.7,7.12)(-3.54,11.36)
\psline[linestyle=dotted](2.83,7.83)(-0,10.65)
\psline[linestyle=dotted](0.7,8.53)(2.12,9.95)
\psline[linestyle=dotted](0.7,8.53)(-2.12,11.36)
\psline[linestyle=dotted](-0,9.24)(1.41,10.65)
\psline[linestyle=dotted](-1.42,7.83)(-0,9.24)
\psline[linestyle=dotted](-1.42,7.83)(-4.25,10.65)
\psline[linestyle=dotted](-1.42,9.24)(-0,10.65)
\psline[linestyle=dotted](-2.83,10.65)(-2.12,11.36)
\psdots[linecolor=darkgray](0,5.71)
\psdots[linecolor=darkgray](-1.41,8.54)
\psdots[linecolor=darkgray](-0.71,9.24)
\psdots[linecolor=darkgray](0.71,7.83)
\psdots[linecolor=darkgray](-1.41,9.95)
\psdots[linecolor=darkgray](-2.83,11.36)
\psdots[linecolor=darkgray](0.71,9.24)
\psdots[linecolor=darkgray](0,9.95)
\psdots[linecolor=darkgray](-0.71,10.66)
\psdots[linecolor=darkgray](0.71,10.66)
\psdots[linecolor=darkgray](1.41,9.95)
\psdots[linecolor=darkgray](2.12,9.24)
\psdots[linecolor=darkgray](2.83,8.54)
\psdots[linecolor=blue](1,3)
\rput[bl](-5.56,15.23){\blue{$V$}}
\psdots[linecolor=blue](1,0)
\rput[bl](-5.56,15.23){\blue{$W$}}
\rput[bl](-5.56,15.23){$h_1$}
\psdots[linecolor=blue](7,3)
\rput[bl](-5.56,15.23){\blue{$Z$}}
\rput[bl](-5.56,15.23){$i_1$}
\psdots[linecolor=blue](0,1)
\rput[bl](-5.56,15.23){\blue{$A_1$}}
\psdots[linecolor=blue](4,1)
\rput[bl](-5.56,15.23){\blue{$B_1$}}
\rput[bl](-5.56,15.23){$j_1$}
\psdots[linecolor=blue](2,2)
\rput[bl](-5.56,15.23){\blue{$C_1$}}
\psdots[linecolor=blue](2,0)
\rput[bl](-5.56,15.23){\blue{$D_1$}}
\rput[bl](-5.56,15.23){$k_1$}
\psdots[linecolor=blue](6,2)
\rput[bl](-5.56,15.23){\blue{$E_1$}}
\rput[bl](-5.56,15.23){$l_1$}
\psdots[linecolor=xdxdff](3,2)
\rput[bl](-5.56,15.23){\xdxdff{$F_1$}}
\psdots[linecolor=blue](3,0)
\rput[bl](-5.56,15.23){\blue{$G_1$}}
\rput[bl](-5.56,15.23){$m_1$}
\psdots[linecolor=blue](3,4)
\rput[bl](-5.56,15.23){\blue{$H_1$}}
\rput[bl](-5.56,15.23){$n_1$}
\psdots[linecolor=blue](7,4)
\rput[bl](-5.56,15.23){\blue{$I_1$}}
\rput[bl](-5.56,15.23){$p_1$}
\psdots[linecolor=xdxdff](4,3)
\rput[bl](-5.56,15.23){\xdxdff{$J_1$}}
\rput[bl](-5.56,15.23){$q_1$}
\psdots[linecolor=xdxdff](6,3)
\rput[bl](-5.56,15.23){\xdxdff{$K_1$}}
\rput[bl](-5.56,15.23){$r_1$}
\end{pspicture*}\label{fig:poset}
\caption{Here we have the poset for the example given in Figure~\ref{fig:main}.  Dotted lines indicate the boundaries of the original faces, rotated so that the northwest corner is at the lowest point of the poset.}
\end{figure}

\begin{lemma}
  Let $\mu_L=\mu_{\mathcal{F}_L}$ denote the M\"obius function for $\mathcal{F}_L$, with the partial order $\leq_L$. Then for any two faces $F_1=F(B_1)$ and $F_2=F(B_2)$ of $G_L$, we have
 \begin{equation*}
\mu_L(F_1,F_2)=
\begin{cases}
1 & \text{ if $F_1=F_2$ or $B_2\in\mathcal{IC}(D_{F_1})$}\\
-1 & \text{ if $B_2\in\mathcal{OC}(D_{F_1})$}\\
0 & \text{ otherwise.}
\end{cases}
\end{equation*}
\end{lemma}

\begin{proof}
We see that our M\"obius function $\mu_{L}$ has the following interpretation.  For a fixed face $F_1$, we assign to each hook $H(F_2)$ the quantity $\mu_L(F_1,F_2)$.  That is, we count the faces lying under $H(F_2)$ with signed multiplicity $\mu_L(F_1,F_2)$.  By the definition of a M\"obius function, this means we want the total count for a face $F$ to be exactly one if $F=F_1$ and zero if $F\neq F_1$.  The proof is then completed by a simple inclusion-exclusion argument, which is left to the reader.
\end{proof}

To avoid unwieldy notation, we will write $M(B)$ and $M'(B)$ in place of $M(B,\M_L)$ and $M'(B,\M_L)$ when the appropriate $\Le$-diagram $L$ is clear from context.

\begin{thm} \label{thm:main-mobius} Suppose $x\in \positL$.  Then the $\Le$-coordinates of $x$ are the entries of the $\Le$-tableau $T(x)\in\T_L$ defined below.  That is, $\Meas_L(T(x))=x$, and $T(x)$ is the unique $\Le$-tableau whose image under $\Meas_L$ is $x$.
\begin{equation}
T(x)_B=
\begin{cases}
0 & \text{if $P_{M(B)}(x)=0$;}\\
\prod_{L_{C}=+}\left(\frac{P_{M(C)}(x)}{P_{M'(C)}(x)}\right)^{\mu(B,C)} & \text{if $P_{M(B)}(x)\neq 0$.}
\end{cases}
\label{eq:le-tableaux}
\end{equation}
\end{thm}

\begin{proof}
By Theorem~\ref{thm:postnikov-bijection}, there exists a unique $\Le$-tableau $T$ satisfying $\Meas_L(T)=x$.  Here we show that $T$ must be the $\Le$-tableau $T(x)$ defined above.  Suppose that $T$ satisfies $P_J(\Meas_L(T))=P_J(x)$ for all $J \in {[n]\choose k}$.  By Theorem~\ref{thm:le-diag}, if $P_{M(B)}(x)=0$, we must have $L_B=0$, and therefore $T_B=0$.  Whenever $L_B=+$, we can easily see that the ratio \[\frac{P_{M(B)}(\Meas_L(T))}{P_{M'(B)}(\Meas_L(T))}\] is the product of the weights of all faces under the $B$ hook in the $\Gamma$-network $N_T$, each with multiplicity one.  By assumption, we have \[\frac{P_{M(B)}(\Meas_L(T))}{P_{M'(B)}(\Meas_L(T))}=\frac{P_{M(B)}(x)}{P_{M'(B)}(x)}.\]  Since the weight of a hook is simply the product of the weights of faces below the hook, a multiplicative version of M\"obius inversion implies that the weight of the face whose northwest corner is $B$ is given by the ratio \[\prod_{L_{C}=+}\left(\frac{P_{M(C)}(x)}{P_{M'(C)}(x)}\right)^{\mu(B,C)}.\]
Since the positive entries of $T$ are simply the weights of the faces in $N_T$, the entries of $T$ must be those of $T(x)$ given in equation (\ref{eq:le-tableaux}).
\end{proof}

\section{$\Le$-coordinates of a positroid cell in terms of a minimal set of Pl\"ucker coordinates}\label{sec:minimal}

By Theorem~\ref{thm:postnikov-bijection}, the dimension of a positroid cell $\positL$ is $|L|$, the number of ``+'' entries in the corresponding $\Le$-diagram $L$.  However, finding the $\Le$-coordinates of a point $x\in\positL$ via equation~(\ref{eq:le-tableaux}) may require roughly twice this many Pl\"ucker variables.  In this section, we give a formula for the map from $\positL$ to $\T_L$, using precisely $|L|$ Pl\"ucker variables.  This formula is, of course, equivalent to our first formula modulo Pl\"ucker relations, but we now use exactly the desired number of parameters.

Suppose $x\in\positL$ and $\Meas_L(T)=x$.  For a box $B$ in $\lambda$ with $L_B=+$, let $F=F(B)$ be the corresponding face in the $\Gamma$-network $N_T$. Let $U_F$ denote the $F$ hook.  Let $D_F$ denote the unique generalized path lying weakly under $U_F$ which forms the southeastern boundary of $F$.  Let $U_F'$ and $D_F'$ be the northwest-most generalized paths lying strictly southeast of $U_F$ and $D_F$, respectively.

For a generalized path $W$ in a $\Gamma$-network $N$ and a box $B$ in $\lambda$, we set
\begin{equation*} \label{eq:le-tableau-min}
\varepsilon_W(B)=
\begin{cases}
1 & \text{ if $B\in \mathcal{OC}(W)$;}\\
-1 & \text{ if $B\in \mathcal{IC}(W)$;}\\
0 & \text{ otherwise.}
\end{cases}
\end{equation*}
\begin{thm}\label{thm:main-minimal}
Suppose $x\in\positL$ and $\Meas_L(T)=x$.  Then the $\Le$-coordinates of $x$ may be written in the alternate form
\begin{equation}
T_B=
\begin{cases}
0 & \text{if $P_{M(B)}(x)=0$;}\\
\prod_{L_{C}=+}(P_{M(C)}(x))^{\varepsilon(C)} & \text{if $P_{M(B)}(x)\neq 0$,}
\end{cases}
\label{eq:le-tableaux}
\end{equation}
where $\varepsilon(C)=[\varepsilon_{U_F}(C)-\varepsilon_{U_F'}(C)]-[\varepsilon_{D_F}(C)-\varepsilon_{D_F'}(C)]$.
\end{thm}

Before proving Theorem~\ref{thm:main-minimal}, we first state one nearly immediate corollary using the notion of a \emph{totally positive base} given in \cite{FZ1999}.

\begin{cor}
The set of Pl\"ucker coordinates
\[\P_L=\{P_{M(B)}:L_B=+\}\]
forms a \emph{totally positive base} for the non-vanishing Pl\"ucker coordinates $\{P_J:J\in\M_L\}$ of the positroid cell $\positL$.  That is, every Pl\"ucker coordinate $P_J$ with $J\in\M_L$ can be written as a subtraction-free rational expression (i.e., a ratio of two polynomials with nonnegative integer coefficients) in the elements of $\P_L$, and $\P_L$ is a minimal (with respect to inclusion) set with this property.  Further, each $P_J$ with $J\in \M_L$ is a Laurent polynomial in the elements of $\P_L$, with nonnegative coefficients.
\end{cor}

\begin{proof}
Suppose $x\in\positL$, with $\Meas(T)=x$.  By Theorem~\ref{thm:main-minimal}, every face weight of the $\Gamma$-network $N_{T}$ can be written as a monomial rational expression in the elements of $\P_L$.  Each Pl\"ucker coordinate $P_J(x)$ is a sum of products of face weights, by Definition~\ref{def:gamma-meas-map}.  It is then clear that each $P_J$ is a Laurent polynomial with nonnegative coefficients in elements of $\P_L$. It is easily verified that the elements of $\P_L$ are algebraically independent, so that $\P_L$ is minimal.
\end{proof}

To prove Theorem~\ref{thm:main-minimal}, we will need the following technical lemma, which gives the weights of certain nested path families.

\begin{lemma}\label{lem:nestedfamilyweight}
Suppose $T$ is a $\Le$-tableau with corresponding $\Gamma$-network $N_T$. Let $W$ be a generalized path in $N_T$. Let $\Nest(W)$ denote the northwest-most non-intersecting path family lying weakly southeast of $W$.  Then
\begin{equation}\label{eq:nestedfamilyweight}
\wt(\Nest(W))=\prod_{L_C=+} \left(P_{M(C)}(\Meas(T))\right)^{\varepsilon_W(C)}.
\end{equation}
\end{lemma}

\begin{figure}[ht]
\psset{unit=0.6cm,dotstyle=*,dotsize=3pt 0,linewidth=0.8pt,arrowsize=3pt 2,arrowinset=0.25}
\begin{pspicture*}(0,-0.5)(19,6.1)
\psline(1,1)(1,6)
\psline(1,6)(8,6)
\psline(8,6)(8,3)
\psline(8,3)(7,3)
\psline(7,3)(7,2)
\psline(7,2)(5,2)
\psline(5,2)(5,1)
\psline(5,1)(1,1)
\psline[linewidth=2pt](8,5)(4,5)
\psline[linewidth=2pt](4,5)(4,4)
\psline[linewidth=2pt](4,4)(2,4)
\psline[linewidth=2pt](2,4)(2,2)
\psline[linewidth=2pt](2,2)(1,2)
\psline[linewidth=2pt](1,2)(1,1)
\psline[linewidth=2pt](8,4)(5,4)
\psline[linewidth=2pt](5,4)(5,3)
\psline[linewidth=2pt](5,3)(3,3)
\psline[linewidth=2pt](3,3)(3,1)
\psline[linewidth=2pt](5,2)(4,2)
\psline[linewidth=2pt](4,2)(4,1)
\psline[linewidth=2pt,linestyle=dotted](2,1.86)(2,1)
\psline[linewidth=2pt,linestyle=dotted](2,1.86)(5,1.84)
\psline(10,1)(10,6)
\psline(10,6)(17,6)
\psline(17,6)(17,3)
\psline(17,3)(16,3)
\psline(16,3)(16,2)
\psline(16,2)(14,2)
\psline(14,2)(14,1)
\psline(14,1)(10,1)
\psline[linewidth=2pt](17,5)(13,5)
\psline[linewidth=2pt](13,5)(13,4)
\psline[linewidth=2pt](13,4)(11,4)
\psline[linewidth=2pt](11,4)(11,1)
\psline[linewidth=2pt](17,4)(14,4)
\psline[linewidth=2pt](14,4)(14,3)
\psline[linewidth=2pt](14,3)(12,3)
\psline[linewidth=2pt](12,3)(12,1)
\psline[linewidth=2pt](14,2)(13,2)
\psline[linewidth=2pt](13,2)(13,1)
\psline[linewidth=2pt,linestyle=dotted](10.1,1.86)(14,1.83)
\psline[linewidth=2pt,linestyle=dotted](10.1,1.86)(10.1,1)
\uput[r](8,5){$W_1$}
\uput[r](8,4){$W_2$}
\uput[u](5.25,1.8){$W_3$}
\uput[d](2,1){$\widehat{W}_1$}
\uput[r](17,5){$\overrightarrow{W_1}$}
\uput[r](17,4){$\overrightarrow{W_2}$}
\uput[u](14.25,1.8){$\overrightarrow{W}_3$}
\uput[d](10,1){$\overleftarrow{W_1}$}
\end{pspicture*}
\caption{Finding the weight of a nested path family.}\label{fig:nestedfamilyweight}
\end{figure}
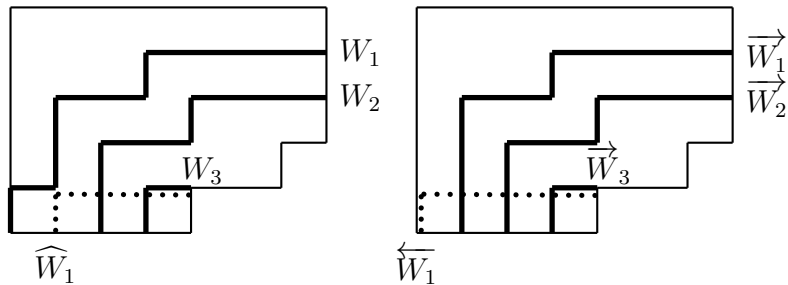

\begin{proof}
We proceed by induction on the number of outer corners of $W$.  If $W$ has a single outer corner, the result follows from the definition of $M(B)$.  Otherwise, assume $W$ has $\ell$ outer corners and split $W$ as follows: let $\overrightarrow{W}$ be the path determined by the first $\ell-1$ outer corners of $W$ (ordered from northeast to southwest) and let $\overleftarrow{W}$ be the hook determined by the last outer corner of $W$.  If $\overrightarrow{W}$ and $\overleftarrow{W}$ do not intersect, the result clearly holds.  (This can happen when $\lambda$ is not the full $k\times n$ rectangle.)  Otherwise, let $\widehat{W}$ be the hook determined by the inner corner of $W$ which is between the last two outer corners of $W$.

Now, $\Nest(W)$ is a disjoint union of paths in $N_T$.  Write $\Nest(W)$ as the ordered collection of path families $(W_1,W_2,\ldots)$, where a path $Y$ in $\Nest(W)$ lies in the block $W_i$ if exactly $i$ paths of $\Nest(W)$ lie strictly northwest of $Y$.  (For $i$ large enough, $W_i$ will be empty.  Recall that the weight of an empty path collection is $1$.)  Write $\Nest(\overrightarrow{W})$, $\Nest(\overleftarrow{W})$, and $\Nest(\widehat{W})$ in the same manner.

We claim that for each $i$, $\wt(\overrightarrow{W}_i)\wt(\overleftarrow{W}_i)=\wt(W_i)\wt(\widehat{W}_i)$.  More precisely, let $(v_1, \ldots, v_m)$ be the vertices at which $\overrightarrow{W}_i$ and $\overleftarrow{W}_i$ intersect.  Then we claim that $W_i$ is the path along edges of $\overrightarrow{W}_i$ or $\overleftarrow{W}_i$ which starts at the source of $\overrightarrow{W}_i$ and takes the northwest-most path between each $v_m$ and $v_{m+1}$ and $\widehat{W}_i$ is the path which starts at the source of $\overleftarrow{W}_i$ and takes the southeast-most path between each $v_m$ and $v_{m+1}$.  This is clearly true for $i=1$.  The remainder, which depends on the $\Le$-property, is left as an exercise for the reader.

Since the weight of a path family is the product of the weights of the individual paths, we then have
\[
\wt(\Nest(W))=\frac{\prod_B \left(P_{M(B)}(\Meas(T))\right)^{\varepsilon_{\overrightarrow{W}}(B)}\cdot\prod_B \left(P_{M(B)}(\Meas(T))\right)^{\varepsilon_{\overleftarrow{W}}(B)}}{\prod_B \left(P_{M(B)}(\Meas(T))\right)^{\varepsilon_{\widehat{W}}(B)}},
\]
which is precisely equation~(\ref{eq:nestedfamilyweight}), since $\overleftarrow{W}$ has a single outer corner (which is an outer corner of $W$) and no inner corners, and $\widehat{W}$ has a single outer corner (which is an inner corner of $W$) and no inner corners.
\end{proof}

\begin{proof}[Proof of Theorem~\ref{thm:main-minimal}:]
 Suppose $W$ is a generalized path in $N_T$. Let $W'$ be the northwest-most generalized path lying strictly below $W$.  We can easily see that the ratio $\frac{\wt(\Nest(W))}{\wt(\Nest(W'))}$ is the product of the weights of the faces lying under $W$, each with multiplicity one, since the weight of each face appearing in this ratio occurs exactly one more time in $\wt(\Nest(W))$ than it does in $\wt(\Nest(W'))$.

Then, since $U_F$ and $D_F$ bound precisely the face $F=F(B)$, the face weight $T_B$ must be given by the ratio
\[\left(\frac{\wt(\Nest(U_F))}{\wt(\Nest(U_F'))}\right)/ \left(\frac{\wt(\Nest(D_F))}{\wt(\Nest(D_F'))}\right).\] Combining this with equation (\ref{eq:nestedfamilyweight}) then yields
 equation~(\ref{eq:le-tableau-min}), since we require that $P_J(\Meas(T))=P_J(x)$ for all $J\in{[n]\choose k}$.
\end{proof}

\section*{Acknowledgements}
The author would like to thank Sergey Fomin and Lauren Williams for many helpful conversations and an anonymous FPSAC reviewer for useful comments on an early version of the manuscript.

\end{document}